\documentclass[12pt]{amsart}
\usepackage[left=1in,top=1in,right=1in,bottom=1in,letterpaper]{geometry}
\usepackage{amsmath}
\usepackage{amssymb}
\usepackage{amscd}
\usepackage{url}
\usepackage{verbatim} 
\usepackage{xcolor}
\usepackage{aliascnt}		
\usepackage{amsthm}
\usepackage{pifont}
\usepackage{DotArrow}
\usepackage{graphicx}
\usepackage{libertine}
\usepackage{hyperref}

\title{Renormalizing An Infinite Rational IET}

\author[]{W. Patrick Hooper}
\address{The City College of New York, New York, NY, USA 10031}
\address{CUNY Graduate Center, New York, NY, USA 10016}
\email{whooper@ccny.cuny.edu}

\author[]{Kasra Rafi}
\address{University of Toronto, Toronto, ON, Canada M5S 2E4}
\email{rafi@math.toronto.edu}

\author[]{Anja Randecker}
\address{University of Toronto, Toronto, ON, Canada M5S 2E4}
\email{anja@math.toronto.edu}

\ifdefined\theorem
\else
\newtheorem{theorem}{Theorem}

\newaliascnt{proposition}{theorem}  
\newtheorem{proposition}[proposition]{Proposition}
\aliascntresetthe{proposition}

\newaliascnt{lemma}{theorem}  
\newtheorem{lemma}[lemma]{Lemma}
\aliascntresetthe{lemma}

\newaliascnt{remark}{theorem}  
\newtheorem{remark}[remark]{Remark}
\aliascntresetthe{remark}

\newaliascnt{corollary}{theorem}  
\newtheorem{corollary}[corollary]{Corollary}
\aliascntresetthe{corollary}

\theoremstyle{definition}

\newaliascnt{definition}{theorem}  

\aliascntresetthe{definition}

\fi
\ifdefined\openquestion
\else
\newtheorem{openquestion}[theorem]{Open Question}
\fi

\ifdefined\question
\else
\newtheorem{question}[theorem]{Question}
\fi

\def\sA{{\mathcal{A}}}
\def\sC{{\mathcal{C}}}
\def\sN{{\mathcal{N}}}

\def\N{\mathbb{N}}
\def\Q{\mathbb{Q}}
\def\R{\mathbb{R}}
\def\Z{\mathbb{Z}}

\begin{document}
\clearpage
\begin{abstract}
We study an interval exchange transformation of $[0,1]$ formed by cutting the interval at the points $\frac{1}{n}$ and reversing the order of the intervals. We find that the transformation is periodic away from a Cantor set of Hausdorff dimension zero. On the Cantor set, the dynamics are nearly conjugate to the $2$--adic odometer. 
\end{abstract}
\maketitle
\thispagestyle{empty}

\section*{Introduction}

We study variations of the following interval exchange transformation: Consider the interval $[0,1)$ and cut it into subintervals of the form $[1-\frac{1}{k},1-\frac{1}{k+1})$ for integers $k \geq 1$. We are interested in the dynamical system $T_1 : [0,1) \to [0,1)$ that reverses the order of the intervals, see \autoref{fig:iet}.

To study this map $T_1$, we are also interested in similar maps $T_N$ on particular subintervals $X_N \subset [0,1)$. For this, let $N$ be a positive integer and let $X_N$ denote the half-open interval $[0,\frac{1}{N})$. Now consider the dynamical system $T_N : X_N \to X_N$ where $X_N$ is cut into half-open intervals of the form $[\frac{1}{N}-\frac{1}{k},\frac{1}{N}-\frac{1}{k+1})$ for $k\geq N$. Reversing the order of these intervals can be described by applying a translation by $\frac{1}{k} + \frac{1}{k+1}-\frac{1}{N}$ to each such interval. 
More formally, the map $T_N : X_N \to X_N$ is defined by
$$T_N(x)=x - \frac{1}{N} + \frac{1}{k} + \frac{1}{k+1} \quad \text{where} \quad k=\left \lfloor \frac{1}{\frac{1}{N}-x}\right \rfloor.$$
Here $\lfloor \star \rfloor$ denotes the greatest integer less than or equal to $\star$.
The map $T_N$ is nearly a bijection: it is one-to-one and its image is the open interval $(0,\frac{1}{N})$. 

\begin{figure}[b]
\includegraphics[width=\textwidth]{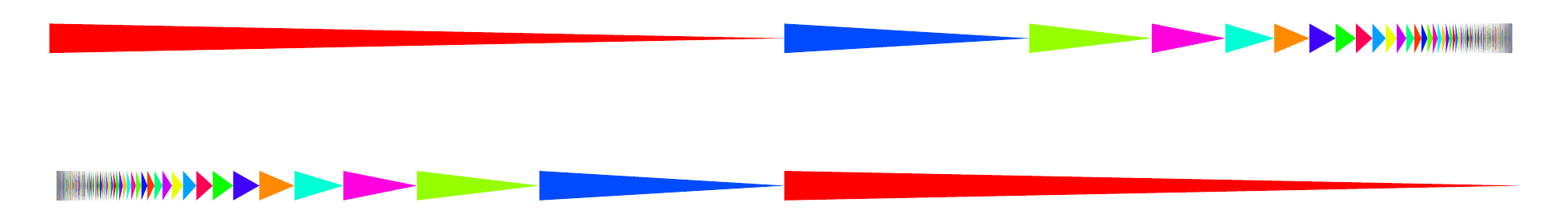}
\caption{{\em Top:} The interval $[0,1)$ cut into intervals of the form $[1-\frac{1}{k},1-\frac{1}{k+1})$. {\em Bottom:} The images of these intervals under $T_1$.}
\label{fig:iet}
\end{figure}

Following notation that is standard in the theory of dynamical systems, we use $T_N^j(x)$ to indicate the point that is obtained by applying this map $j$ times to the point $x \in X_N$. A point $x$ is called {\em periodic} under $T_N$ if there exists an integer $j >0$ such that $T_N^j(x)=x$. We will show:

\begin{theorem}
\label{thm: zero dimensional}
For each positive integer $N$, there is a Cantor set $\bar \Lambda_N \subset [0,\frac{1}{N}]$ of Hausdorff dimension zero such that $x$ is periodic under $T_N$ if and only if there exists an $\epsilon>0$ such that $(x,x+\epsilon) \cap \bar \Lambda_N = \emptyset$. In particular, $x$ is periodic if $x \not \in \bar \Lambda_N$, so the vast majority of points are periodic under the map $T_N$.
\end{theorem}

Let $\Lambda_N$ denote the set of points which are {\em aperiodic} (not periodic) under $T_N$.
The dynamics of the restriction of $T_N$ to $\Lambda_N$ turn out to be related to the $2$--adic odometer which we now define.

Let $\sA$ be the alphabet $\{0,1\}$ and $\N=\{0,1,2,\ldots\}$. 
The $2$--adic integers are the set of formal sums
\begin{equation}
\label{eq:2-adics}
\sum_{k \in \N} \alpha_k 2^k \quad \text{with $\alpha_k \in \sA$ for all $k$.}
\end{equation}
We identify the $2$--adic integers with the space $\sA^\N$ consisting of all sequences $\alpha=(\alpha_0, \alpha_1, \ldots)$ with each $\alpha_k \in \sA$.
The $2$--adic integers form an abelian group with the operation of addition allowing carrying of the form $1\cdot 2^k + 1\cdot 2^k=1\cdot 2^{k+1}$. The \emph{addition-by-one map} is given by adding $1 \cdot 2^0$ to a $2$--adic integer. In terms of sequences, the addition-by-one map is the map $f:\sA^\N \to \sA^\N$ defined by
\begin{equation}
\label{eq:f}
f(\alpha)_k = \begin{cases}
0 & \text{if $k<j$} \\
1 & \text{if $k=j$} \\
\alpha_k & \text{if $k>j$}
\end{cases}
\quad \text{where} \qquad 
j=\min~(\{k:~\alpha_k=0\} \cup \{+\infty\}).
\end{equation}
This map is also called the {\em $2$--adic odometer.} It is a homeomorphism when we equip $\sA$ with the discrete topology and $\sA^\N$ with the product topology. It is well known that $f$ is {\em minimal} (all orbits are dense) and {\em uniquely ergodic} (there is only one invariant Borel probability measure) \cite[~\S 1.6.2]{Fogg}.

Let $\sN$ be the set of all $2$--adic integers $\alpha \in \sA^\N$ which end in an infinite sequence of ones, i.e.,
$$\sN=\{\alpha \in \sA^\N: \quad \text{there exists a $K$ such that $\alpha_k=1$ for $k>K$}\}.$$
Another characterization of this set is as the set of $2$--adic integers $\alpha$ such that there exists an $n>0$ for which $f^n(\alpha)=\overline{0}$, where $\overline{0} \in \sA^\N$ is the zero element defined by $\overline{0}_k=0$ for all $k$.

We show that the restriction of $T_N$ to the aperiodic set $\Lambda_N$ mirrors the action of the $2$--adic odometer:
\begin{theorem}
\label{thm:aperiodic}
For each positive integer $N$ and $T=T_N$, there is a continuous bijection $h=h_N$ from $\sA^\N \smallsetminus \sN$ to the aperiodic set $\Lambda_N \subset X_N$ such that
$T \circ h(\alpha)=h \circ f(\alpha)$ for all $\alpha \in \sA^\N$.
\end{theorem}

We give an explicit description of the aperiodic set $\Lambda_N$ and an explicit description of the map~$h$ in \autoref{sect:conjugacy}.

The {\em least period} of a periodic point $x \in X_N$ under $T_N$ is the smallest $k>0$ such that $T_N^k(x)=x$.
An interesting question this work leaves open is (see also \autoref{rem:periods}):
\begin{question}
Which integers $p>0$ appear as least periods of 
periodic points under $T_N$? 
For each such $p$ what is the Lebesgue measure of the set of periodic points of least period $p$?
\end{question}

\subsection*{Connections to other work}
Another infinite interval exchange transformation (IET) is given by the {\em Van der Corput map}:
\begin{equation}
\label{eq:Reza}
S:[0,1) \to [0,1); \quad x \mapsto x-1+2^{k}+2^{k-1} \quad \text{where} \quad
k=\lfloor \log_2(1-x) \rfloor.
\end{equation}
This map is nearly conjugate to the $2$--adic odometer; see discussions in \cite[\S 5.2.3]{Fogg}, \cite[~\S 3.8]{Silva} and \cite[\S 2]{LT16}. This map turns out to be semi-conjugate to the restriction of $T_N$ to $\Lambda_N$ as described in \autoref{thm:aperiodic}.

Polygon and polytope exchange transformations (PETs) are higher dimensional analogs of~IETs.
There are numerous examples in the literature of such maps admitting an open dense set of periodic points but with interesting dynamics on the complimentary sets. See for example
\cite{AkiyamaHarriss},
\cite{Goetz00b},
\cite{Goetz03}, 
\cite{H12},
\cite{Schwartz14},
\cite{Yi}.
This sort of behaviour is impossible for IETs formed by permuting finitely many intervals \cite[Theorem 6.6]{MT}. Part of the purpose of this article is to illustrate that this phenomenon arises in natural infinite IETs.

It is not the case that every infinite IET has a minimal component where the restriction of the map to this component is conjugate to an odometer. 
For example, there exists an infinite minimal IET of $[0,1]$ with positive entropy such that all lengths are $2$--adic rationals (see \cite[\S 4]{DHV}) but odometers have entropy zero.

\section{Generalities}

\subsection*{Interval exchanges}
For us, an {\em interval exchange transformation} (IET) is a one-to-one piecewise translation $T:X \to X$ where $X \subset \R$ is a bounded interval. That is, we have a partition of $X$ into countably many subintervals $X=\bigsqcup_{j \in J} I_j$ and a choice of translations $\tau_j \in \R$ for $j \in J$ such that the map
$$T:X \to X; \quad x \mapsto x+\tau_j \quad \text{when $x \in I_j$}$$
is injective.

We call $T$ {\em rational} if each $\tau_j$ lies in $\Q$. The following is a classical observation:

\begin{proposition}
\label{prop:periodicity}
If $\,T$ is a rational IET and $\tau_j$ takes only finitely many values, then every orbit of $\,T$ is periodic. More generally, if $\,T:X \to X$ is a rational IET and $x \in X$, then $x$ has a periodic orbit unless
$$\{\tau_j:~ \text{there is an $n\geq 0$ such that $T^n(x) \in I_j$}\} \quad \text{is infinite.}$$
\end{proposition}
\begin{proof}
Since each $\tau_j \in \Q$ and there are only finitely many translations $\tau_j$, there is a $d \in \Q$ such that $\frac{\tau_j}{d} \in \Z$ for all $j$. Observe that $T$ permutes the finitely many points in $(x+d \Z)\cap X$.
\end{proof}

When we were working on this project, we wondered how common it is to have a dense set of periodic points for a
rational IET which is {\em infinite} in the sense that $\{\tau_j\}$ is infinite. Some experimental work of Anna Tao (undergraduate, CCNY) seems to suggest that this sort of periodicity is rare. However, we still wonder if there are natural classes
of infinite rational IETs in which having a dense set of periodic points is typical.

At this point, there are a number of infinite rational IETs in the literature. 
Equation \eqref{eq:Reza} gives an infinite rational IET without periodic orbits, and there are other examples corresponding to $p$--adic odometers and the Chacon middle third transformation \cite[\S 3]{Downarowicz} \cite{LT16}. One way to get such a rational IET is from straight-line flows in directions of rational slope on an infinite-type translation surface all of whose saddle connections have holonomy in $\Q^2$. Symmetric surfaces of this form have been described in \cite{Chamanara04}, \cite{Bowman13} and \cite{LT16}.

\subsection*{Return maps}
If $Y \subset X$ is an interval, the {\em first return time} of $y\in Y$ to $Y$ is
$$r(y) = \min~\big(\{n > 0:~T^n(y) \in Y\} \cup \{+\infty\}\big).$$
Assuming $r<+\infty$ on $Y$, we define the {\em first return map} $\hat T:Y \to Y$ to be the map
$$\hat T:Y \to Y; \quad y \mapsto T^{r(y)}(y).$$
If $T$ is an IET in the sense above, then so is $\hat T$. Furthermore, $\hat T$ is rational whenever $T$ is rational.

\section{Basic return maps}
Here we prove some basic results about the maps $T_N:X_N \to X_N$ defined in the introduction.
First we fully describe the return map to $X_{N(N+1)}$.

\begin{lemma}
\label{lem:return map}
For any $N$, the first return map of $\,T_N$ to the interval $X_{N(N+1)}=[0,\frac{1}{N(N+1)})$ is given by $T_{N(N+1)}$. Furthermore, the return time is $2$ on all of $X_{N(N+1)}$.
\end{lemma}
\begin{proof}
For each $x \in X_{N(N+1)}$ we see that $k=\lfloor{\frac{1}{1/N-x}}\rfloor=N$ and thus $T_N(x)=x+\frac{1}{N+1}$. This shows $T_N(X_{N(N+1)})=[\frac{1}{N+1},\frac{1}{N})$ and in particular, no point has least period $1$. We have that 
$$\left[\frac{1}{N+1},\frac{1}{N}\right)=\bigcup_{\ell \geq N(N+1)} \left[\frac{1}{N}-\frac{1}{\ell},\frac{1}{N}-\frac{1}{\ell+1}\right).$$
Set $x'=T_N(x) \in [\frac{1}{N+1},\frac{1}{N})$ and $k'=\lfloor{\frac{1}{1/N-x'}}\rfloor \geq N(N+1)$. 
We compute 
\begin{equation}
\label{eq:square}
T_N^2(x)=T_N(x')=x' - \frac{1}{N} + \frac{1}{k'} + \frac{1}{k'+1}
=x - \frac{1}{N(N+1)} + \frac{1}{k'} + \frac{1}{k'+1}.
\end{equation}
Now observing that
$$k'=\left\lfloor \frac{1}{\frac{1}{N}-x'} \right\rfloor=
\left\lfloor \frac{1}{\frac{1}{N}-x-\frac{1}{N+1}} \right\rfloor=
\left\lfloor \frac{1}{\frac{1}{N(N+1)}-x} \right\rfloor,$$
we see from \eqref{eq:square} that $T_N^2(x)$ coincides with $T_{N(N+1)}(x)$.
\end{proof}

We get periodic points as a consequence:

\begin{corollary}
\label{cor:periodic interval}
For any $N$, every point in $[\frac{1}{N(N+1)},\frac{1}{N+1})$ has a periodic orbit under $T_N$.
\end{corollary}
\begin{proof}
Observe that $T_N\big([\frac{1}{N(N+1)},\frac{1}{N+1})\big)=[\frac{1}{N(N+1)},\frac{1}{N+1})$, because $T_N$ reverses the order
of intervals and we already know $T_N\big([0,\frac{1}{N(N+1)})\big)=[\frac{1}{N+1},\frac{1}{N})$ and
$T_N\big([\frac{1}{N+1},\frac{1}{N})\big)=(0,\frac{1}{N(N+1)})$. Moreover, there are only finitely many distinct translations occurring on this interval, namely the translations associated to $(\frac{1}{k+1}, \frac{1}{k}]$ for values of $k$ satisfying $N+1 \leq k <N(N+1)$. 
\autoref{prop:periodicity} then guarantees that every point in $[\frac{1}{N(N+1)},\frac{1}{N+1})$ is periodic.
\end{proof}

\begin{remark}\label{rem:periods}
In the case $N=2$, every point in the interval $[\frac 16, \frac 13)$ has a periodic orbit under~$T_2$ that has least period $10$. In general, however, there may be points in $[\frac{1}{N(N+1)},\frac{1}{N+1})$ that do not have the same least period. 
For example, for $N=3$, points may have least period either $920$ or $930$ under $T_3$.  

To describe more examples for larger $N$, we define another family of IETs
\begin{equation*}
R_{m, n} : \left[\frac 1m, \frac 1n\right) \to \left[\frac 1m, \frac 1n\right)
\end{equation*}
for all $m > n > 0$ by breaking this interval into subintervals of the form 
$[\frac{1}{k+1}, \frac 1k)$ for $m > k \geq n$
and reversing the order of the subintervals. Note that the restriction of $T_N$ to the interval
$[\frac{1}{N(N+1)},\frac{1}{N+1})$ is $R_{N(N+1), N+1}$. In fact, there are many subintervals in $X_1$  of the form $[\frac 1m, \frac 1n)$ that are preserved by a 
power of $T_1$ and where the first return map is $R_{m,n}$.

For example, the interval $[\frac 1{42}, \frac 17)$ is sent to itself by $T_6$. The restriction of $T_6$ to $[\frac 1{42}, \frac 17)$ coincides with $R_{42,7}$ and with the restriction of $T_1^4$ to this interval.
Analyzing the periodic orbits in $[\frac 1{42}, \frac 17)$ with \cite{surface_dynamics}, we see that there are nine different least periods occurring under $T_6$, namely: 
\begin{align*}
272, \qquad 2002,&& 105252, &&125986, &&9515623638834, \\
70542359811724,&&  
35513020871128, &&13883349533760,  &&43184371863572.
\end{align*} 
Furthermore, the interval $[\frac 1{42}, \frac 17)$ itself has subintervals of the same type that are preserved by some power of $T_1$. 
Namely, $[\frac 1{15}, \frac 1{10})$ is sent to itself under $T_1^8$ which coincides with $T_6^2$ and with $R_{15,10}$. 
Also each of the intervals $[\frac 1{18}, \frac 1{15})$, $[\frac 1{24}, \frac 1{18})$ and $[\frac 1{42}, \frac 1{24})$ are sent to themselves under $T_1^{16}$ which coincides with $T_6^4$ and the first return map is of the form $R_{m,n}$. In case of the interval $[\frac 1{42}, \frac 1{24})$, there are again five different least periods occurring.

In all these cases, every possible least period has to be a divisor of the least common multiple of the denominators $n, n+1, \ldots, m$. It would be interesting to classify for which pairs $(m, n)$ every point in the interval $[\frac 1{m}, \frac 1{n})$ has the same least period under $R_{m,n}$. 
\end{remark}

\section{Cantor sets}
\label{sect:cantor}
In this section we work through a general construction of a Cantor set. We will see later in the article that
the set $\bar \Lambda_N$ arises as such a Cantor set.

The {\em free monoid} on the alphabet $\sA=\{0,1\}$ is the set $\sA^\ast$ of all finite sequences in $\sA$
equipped with the binary operation of concatenation. 
An element $w \in \sA$ is called a {\em word} and has a {\em length} $|w| \in \N$ representing the number of elements strung together. We write $\sA^k$ to denote the collection of all $w \in \sA^\ast$
with length $k$.
The unique element $\varepsilon \in \sA^\ast$ with length zero
is called the {\em empty word} and is the unique identity element of the monoid. 

Every element $w \in \sA^\ast$ can be written as
$$w=w_0 w_1 \ldots w_{|w|-1} \quad \text{with each $w_i \in \sA$}.$$
Concatenation is the operation defined by
$$xy = x_0 x_1 \ldots x_{|x|-1} y_0 y_1 \ldots y_{|y|-1}.$$
More formally, $xy$ is defined to be the finite sequence $z$ of length $|x|+|y|$ such that
$$z_i=\begin{cases}
x_i & \text{if $i<|x|$,} \\
y_{i-|x|} & \text{if $|x|\leq i \leq |x|+|y|$}.
\end{cases}$$
We use exponential notation for repeated concatenation so that $w^k$ denotes the concatenation of $k$ copies of $w$.
For example $0^9$ denotes the word $w$ where $|w|=9$ and $w_i=0$ for $i=0, \ldots, 8$.

We now informally describe the Cantor sets that we are interested in. We use a variant of the standard construction of the Cantor ternary set in $\R$, where the Cantor set is obtained by removing the middle third interval of $[0,1]$, then removing the middle third intervals of the remaining segments, and so on.
Our Cantor set is similarly defined as the intersection $\bigcap_{k \geq 0} C_k$ and each~$C_k$ is a finite union of closed intervals.
The sets $C_k$ are defined inductively starting with a single interval $C_0=[a_0,b_0]$ and the set $C_{k+1}$ is formed
by removing middle intervals of equal length from each of the intervals making up $C_k$.
In contrary to the construction of the Cantor ternary set, the ratio of the lengths of intervals making up $C_{k+1}$ to the lengths of intervals making up $C_k$ is not necessarily the same for all $k$.
We denote these ratios by numbers $s_k$.

\begin{figure}
\includegraphics[width=3in]{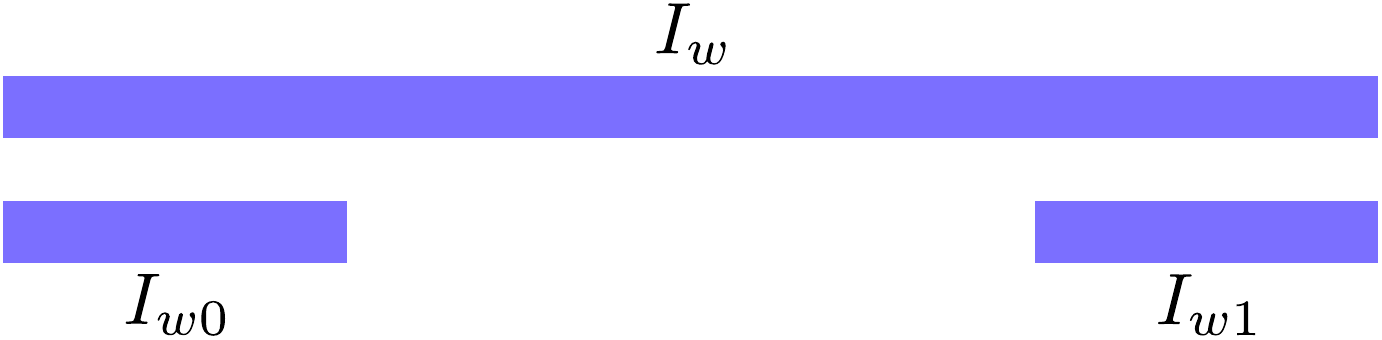}
\caption{The intervals $I_{w0}$ and $I_{w1}$ produced from $I_w$ when $s_{|w|}=\frac{1}{4}$.}
\label{fig:subdivision}
\end{figure}

We now give a more formal construction of our Cantor set. 
Fix an initial interval $[a_0, b_0]$ and a sequence $s=\{s_k\}_{k \in \N}$ of real numbers satisfying 
\begin{equation}
0 < s_k \leq \frac{1}{2} \text{ for all $k \in \N$} \quad \text{and} \quad \limsup s_k<\frac{1}{2}.
\label{eq:conditions on s}
\end{equation}
We inductively define an interval $I_w$ for each $w \in \sA^\ast$. We define $I_\varepsilon = [a_0,b_0]$. Assuming $I_w$ is defined to be $[a,b]$, we define 
\begin{equation}
\label{eq:inductive}
I_{w0}=[a,a+s_{|w|}(b-a)] \quad \text{and} \quad
I_{w1}=[b-s_{|w|}(b-a),b];
\end{equation}
see \autoref{fig:subdivision}.
Observe that if $s_{|w|}<\frac{1}{2}$ then $I_{w0} \cup I_{w1}$ is the interval $I_w$ with the middle open interval removed whose length is $1-2s_{|w|}$ times the length of the whole interval. On the other hand, if $s_{|w|}=\frac{1}{2}$ the intervals
$I_{w0}$ and $I_{w1}$ are formed by cutting $I_w$ at the midpoint. In particular, the length of the interval $I_w$ only depends on $|w|$ and the fixed sequence $s$. The length is given by  $\ell_{|w|}$ where
\begin{equation}
\label{eq:lengths}
\ell_0=b_0-a_0 
\quad \text{and} \quad
\ell_{k} = (b_0-a_0) \prod_{j=0}^{k-1} s_j \quad \text{for $k \geq 1$}.
\end{equation}

We define the Cantor set $\sC=\sC\big(s,[a_0,b_0]\big)$ by defining
$$C_k = \bigcup_{w \in \sA^k} I_w \quad \text{and} \quad \sC = \bigcap_{k \in\N} C_k.$$
It is a standard observation that as long as the sequence $s$ satisfies the conditions in \eqref{eq:conditions on s} that $\sC$ is a Cantor set: it is compact, totally disconnected and perfect.
The following is a standard result on the Hausdorff dimension of $\sC$ (compare \cite[\S 4.10-11]{Matilla}).

\begin{proposition}
\label{prop:Hausdorff dimension}
If $\,\lim_{k \to \infty} s_k=0$ then the Hausdorff dimension of $\sC$ is zero.
\end{proposition}
\begin{proof}
Recall that the {\em $d$--dimensional Hausdorff content} of $\sC$ is 
$$C_H^d(\sC)=\inf \Big\{ \sum_{i} r_i^d:~\text{there is a covering of $\sC$ by balls of radius $r_i > 0$} \Big\}.$$
The {\em Hausdorff dimension} of $\sC$ is $\inf \{d:~C_H^d(\sC)=0\}.$

Fix $d>0$. Now consider an integer $k>0$ and consider that $\bigcup_{w \in \sA^k} I_w$ contains $\sC$. Each interval in the union has length $\ell_k$ and there are $2^k$ words in $\sA^k$, so for this covering $\sum_{i} r_i^d$ yields $2^k (\ell_k/2)^d$. Observe from \eqref{eq:lengths} that
$$\lim_{k \to \infty} 2^k \left({\textstyle \frac{\ell_k}{2}}\right)^d = \left({\textstyle \frac{b_0-a_0}{2}}\right)^d \lim_{k \to \infty} \prod_{j=0}^{k-1}  \left(2 s_j^d\right)$$
and since $s_j \to 0$, this limit is zero. This shows that the $d$--dimensional Hausdorff content is zero for any $d>0$
and so the Hausdorff dimension is zero.
\end{proof}

We can now define the map $h$ that was announced in \autoref{thm:aperiodic} to give a continuous bijection from $\sA^\N \smallsetminus \sN$ to the aperiodic set of $T_N$.
Recall that $\sA^\N$ is the set of $2$--adic integers, consisting of all sequences $\alpha=(\alpha_0, \alpha_1, \ldots)$ with each $\alpha_k \in \sA = \{0,1\}$.
Define the map $h:\sA^\N \to \R$ depending on a sequence $s$ as in \eqref{eq:conditions on s} and on an interval $[a_0,b_0]$ by 
\begin{equation}
\label{eq:h}
h(\alpha)=a_0 + \sum_{k=0}^\infty \alpha_k (\ell_k-\ell_{k+1}).
\end{equation}
We will see in \autoref{lem:cantor set image} that the function $h$ is closely related to our construction of the Cantor set $\sC(s,[a_0,b_0]\big)$.
We will also see in \autoref{sect:conjugacy} that $h$ can be used to give a $2$--adic infinite address to every point in the aperiodic set of $T_N$, and that $h$ describes a semi-conjugacy to the $2$--adic odometer.
But first we observe that $h$ can be used to describe the endpoints of the intervals $I_w$ used in the construction of the Cantor set $\sC(s,[a_0,b_0]\big)$.

\begin{proposition}
\label{prop:interval formula}
For each $w \in \sA^\ast$, we have $I_w=\big[h(w\overline{0}),h(w \overline{1})\big]$, where $w \overline{0}$ and $w \overline{1}$ denote the elements of $\sA^\N$ whose first $|w|$ entries are given by $w$ and whose remaining entries are all zeros or all ones respectively.
\end{proposition}
\begin{proof}
Fix $w$ and let $k=|w|$.
Observe that the lengths of $I_w$ and $\big[h(w\overline{0}),h(w \overline{1})\big]$ match since the length of $I_w$ is $\ell_k$ and 
$$h\left(w \overline{1}\right)-h\left(w \overline{0}\right)=\sum_{j=k}^\infty (\ell_j-\ell_{j+1})=\ell_k$$
since $\lim_{j \to \infty} \ell_j=0$. It follows that checking $I_w=\big[h(w\overline{0}),h(w \overline{1})\big]$
is equivalent to checking that the left endpoint of $I_w$ is $h(w \overline{0})$ or checking that the right endpoint of $I_w$ is $h(w \overline{1})$.

We proceed by induction on the length of the word $w$. Observe that $h(\overline{0})=a_0$ and hence $I_\varepsilon=\big[h(\overline{0}),h( \overline{1})\big]$.
Now suppose that $I_w=[a,b]$, $h(w\overline{0})=a$ and $h(w\overline{1})=b$. 
We have to check that $I_{w0}=\big[h(w0\overline{0}),h(w0 \overline{1})\big]$
and $I_{w1}=\big[h(w1\overline{0}),h(w1 \overline{1})\big]$.
The statement for $I_{w0}$ holds because the left endpoint of $I_{w0}$ coincides with the left endpoint of
$I_w$ by definition in \eqref{eq:inductive}, and by hypothesis we have $a=h(w\overline{0})=h(w0\overline{0})$.
The statement for $I_{w1}$ holds because the right endpoint of $I_{w1}$ coincides with the right endpoint of
$I_w$ by \eqref{eq:inductive}, and by hypothesis we have $b=h(w\overline{1})=h(w1\overline{1})$.
\end{proof}

\begin{lemma}
\label{lem:cantor set image}
The image $h(\sA^\N)$ is the Cantor set $\sC = \sC(s,[a_0,b_0]\big)$. Furthermore, $h$ is one-to-one at all $x \in \sC$ except at those $x$ of the form $x=h(w0\overline{1})=h(w1\overline{0})$ for some $w \in \sA^k$ with $s_k=\frac{1}{2}$. The latter case happens only finitely often and in this case, $h$ is two-to-one at $x$.
\end{lemma}
\begin{proof}
First we show that for any $\alpha \in \sA^\N$ we have $h(\alpha)\in \sC$. We must show $h(\alpha) \in C_k$ for every~$k$. Fix a $k$ and set $w=\alpha_0 \alpha_1 \ldots \alpha_{k-1}$. Then observe that
$$h\left(w \overline{0}\right) \leq h(\alpha) \leq h\left(w \overline{1}\right)$$
which implies $h(\alpha) \in I_w \subset C_k$.

Now suppose $x \in \sC$. We study the number of preimages of $x$ under $h$. 
Observe that for each~$k \geq 0$ there exists a $w \in \sA^k$ such that $x \in I_w$.
We break into two cases. First suppose that for each~$k$ there exists a unique $w \in \sA^k$ such that $x \in I_w$.
Denote each such word by $w^k$.
Observe that $w'$ is an initial word of $w^k$ if and only if $I_{w'} \supset I_{w^k}$. It follows that for $j<k$, $w^j$ is the initial subword of $w^k$ of length $j$.
Then we can unambiguously define $\alpha \in \sA^\N$ by $\alpha_i=w^k_i$ for some $k>i$. 
Now observe that $h(\alpha) \in I_{w^k}$ for each $k$. Since the length of $I_{w^k}$ tends to zero as~$k \to \infty$, we see that $h(\alpha)=x$. 
Finally, suppose $\beta \in \sA^\N$ is distinct from $\alpha$. Then there is a $k$ such that the initial word of length $k$ of $\beta$ differs from $w^k$. 
We see that $h(\beta) \in I_{\beta_0 \ldots \beta_{k-1}}$ but $x$ is not in this interval, so $h(\beta) \neq x$. Thus $h$ is one-to-one at $x$.

If we are not in the first case, then there is a smallest $k$ such that there are two words in $\sA^k$ for which $x$ lies in both the corresponding intervals.
From the argument about initial words in the previous paragraph, we see that because $k$ is smallest, the two words have the same initial words.
That is, the two words must have the form $w0$ and $w1$. Thus we have $x \in I_{w0} \cap I_{w1}$. By~\eqref{eq:inductive} we see that $I_{w0} \cap I_{w1} \neq \emptyset$ if and only if $s_k=\frac{1}{2}$.
And if this intersection is non-empty then the intersection just consists of the midpoint of $I_w$.
In this case, $x$ is the right endpoint of $I_{w0}$ and the left endpoint of $I_{w1}$. 
So by \autoref{prop:interval formula}
we see $x=h(w0\overline{1})=h(w1\overline{0})$. Furthermore, it can be deduced by an inductive application of \eqref{eq:inductive}
that for any $j>0$, we have $w' \in \sA^{k+j}$ and $x \in I_{w'}$ if and only if $w' \in \{w01^j,w10^j\}.$
Then if $\beta \in \sA^\N \smallsetminus \{w0\overline{1}, w1\overline{0}\}$, there is some initial word $w'$ of $\beta$
of length $k+j$ such that $w' \not \in \{w01^j,w10^j\}$ and we have $h(\beta) \in I_{w'}$ but $x$ is not in this interval, so $h(\beta)\neq x$. This shows that $h$ is two-to-one at $x$.
Furthermore, there are only finitely many $k>0$ such that $s_k = \frac{1}{2}$ because of \eqref{eq:conditions on s}, so this case only appears finitely often.
\end{proof}

Recall from the introduction that $\sN=\{w\overline{1}:w \in \sA^\ast\} \subset \sA^\N$. This is an important set for us, and we prove the following.

\begin{proposition}\label{prop:N}\leavevmode
\begin{enumerate}
\item The restriction of $h$ to $\sA^\N \smallsetminus \sN$ is injective.
\item The Cantor set $\sC$ is the closure of $h(\sA^\N \smallsetminus \sN)$.
\item The set $h(\sA^\N \smallsetminus \sN)$ is the set of all $x \in \sC$ such that 
$(x,x+\epsilon) \cap \sC \neq \emptyset$ for all $\epsilon>0$.
\end{enumerate}
\end{proposition}
\begin{proof}
Statement (1) is a consequence of \autoref{lem:cantor set image} since $h$ is one-to-one at all points except
that it is possible that $x=h(w0\overline{1})=h(w1\overline{0})$. But we have $w0\overline{1} \in \sN$.

Since $\sC=h(\sA^\N)$ and $\sC$ is closed by construction, to prove statement (2) we just need to find for each $\alpha \in \sN$ a sequence $\alpha^k \in \sA^\N \smallsetminus \sN$
such that $h(\alpha^k)$ converges to $h(\alpha)$. For each $k$, let $w^k=\alpha_0 \ldots \alpha_{k-1} \in \sA^k$
and define $\alpha^k=w^k \overline{0}$. Then both $h(\alpha)$ and $h(\alpha^k)$ lie in $I_{w^k}$ for each $k$ and the length of $I_{w^k}$ tends to zero so we see that $h(\alpha)=\lim h(\alpha^k)$ as desired.

Finally consider statement (3). 
First suppose that $\alpha \in \sA^\N \smallsetminus \sN$. Then there exists a sequence $k_j \to \infty$ such that $\alpha_{k_j}=0$.
For $j\geq 0$, define $\beta^j \in \sA$ so that the sequence agrees with $\alpha$ except that $\beta^j_{k_j}=1$. Observe that by definition of $h$, we have
$h(\beta^j)>h(\alpha)$ and $\lim h(\beta^j)=h(\alpha)$. This proves that $\big(h(\alpha),h(\alpha)+\epsilon\big)$ intersects $\sC=h(\sA^\N)$ for all
$\epsilon >0$.

On the other hand, suppose that $x \in \sC \smallsetminus h(\sA^\N \smallsetminus \sN)$. We need to show that there exists an $\epsilon>0$ such that $(x,x+\epsilon) \cap \sC = \emptyset$. If $x = h(\overline{1})$ then this is clearly true since $h(\overline{1})$ is the right endpoint of~$I_\varepsilon$ by \autoref{lem:key} and $\sC \subset I_\varepsilon$. 
Otherwise there exists a $w \in \sA^\ast$ such that $x=h(w0\overline{1})$. Furthermore, $h$ is one-to-one at $x$ since otherwise we would have $x=h(w1\overline{0})$ as well which would contradict that $x \not \in h(\sA^\N \smallsetminus \sN)$. 
Setting $k=|w|$ we see therefore that $s_k<\frac{1}{2}$ by \autoref{lem:cantor set image}. Since $x=h(w0\overline{1})$
we see that $x$ is the right endpoint of $I_{w0}$. Let $[a,b]=I_w$. Then we see in the notation of~\eqref{eq:inductive}
that $x=a+s_{k}(b-a)$ and the removed interval $I_w \smallsetminus (I_{w0} \cup I_{w_1})$ is $\big(x,x+(1-2 s_k)(b-a)\big)$ which gives an interval of positive length not intersecting $\sC$ as~required.
\end{proof}

\section{The conjugacy}
\label{sect:conjugacy}
Fix a positive integer $N$ and extend it to a sequence inductively by defining 
$$N_0=N \quad \text{and} \quad  N_{k+1}=N_k(1+N_k) \text{ for all $k \geq 0$.}$$ 
By an inductive application of \autoref{lem:return map} we see:

\begin{corollary}
\label{cor:return}
For each $k$, the first return map of $\,T_N$ to $X_{N_k}$ is $T_{N_k}$. 
\end{corollary}

Set $[a_0,b_0]=[0,\frac{1}{N}]$ and define the sequence $s=\{s_k\}$ by $s_k=\frac{1}{1+N_k}$. With this data, we define the Cantor set $\sC=\sC(s,[a_0,b_0])$ and the map $h:\sA^\N \to \R$ as in \autoref{sect:cantor}. See \autoref{fig:cantor_set} for a sketch of $\sC$ when $N=1$. Observe that this choice of $[a_0, b_0]$ and of $s$ and application of~\eqref{eq:lengths} yields $\ell_0 = \frac{1}{N} - 0 = \frac{1}{N_0}$ and inductively we have
$$\ell_k= \ell_{k-1} \cdot s_{k-1} = \frac{1}{N_{k-1}} \cdot \frac{1}{1+N_{k-1}} = \frac{1}{N_k} \quad \text{for every $k$}.$$
We use this information to define the intervals $I_w$ as before. 

\begin{figure}
\includegraphics[width=4.9in]{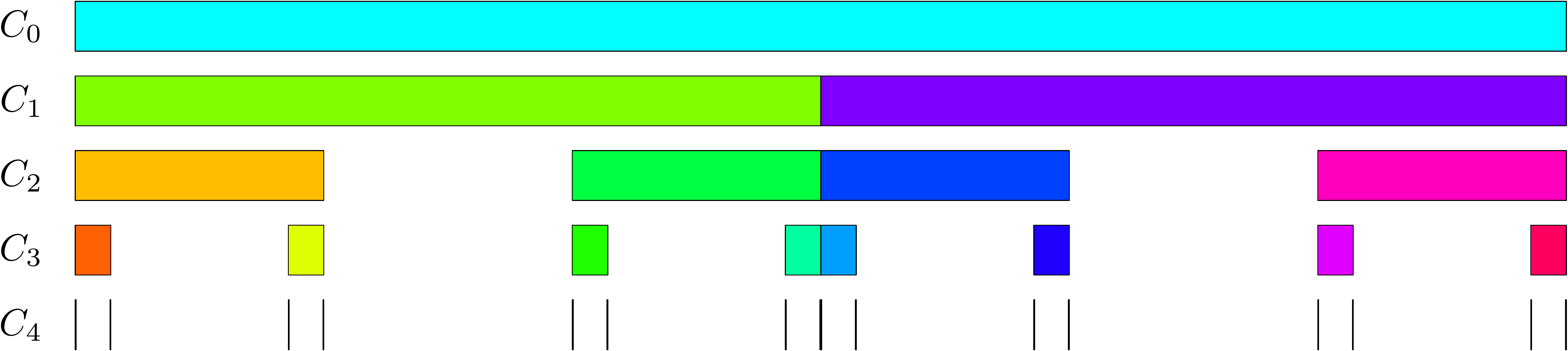}
\caption{The construction of the Cantor set $\sC$ when $N=1$.}
\label{fig:cantor_set}
\end{figure}

Recall the definition of the $2$--adic odometer $f : \sA^\N \to \sA^\N$ in \eqref{eq:f}. We want to extend this addition-by-one map to $\sA^\ast = \cup_{k\geq 0} \sA^k$.
At words of the form $1^k$ for some $k \geq 0$, we leave the map $f$ undefined. 
We define $f:\sA^k \smallsetminus \{1^k\} \to \sA^k$ such that
\begin{equation}
\label{eq:plus one words}
\big(f(w)\big)_i = \begin{cases}
0 & \text{if $i<j$}\\
1 & \text{if $i=j$}\\
w_i & \text{if $i>j$}
\end{cases}
\quad \text{where} \quad
j=\min~\{i:~w_i=0\}.
\end{equation}

For this section, if $I$ is a closed interval, we write $I^\star$ to denote $I$ with its right endpoint removed.

The key to the results announced in the introduction is the following: 

\begin{lemma}
\label{lem:key}
For any $w \in \sA^k \smallsetminus \{1^k\}$, the restriction of $\,T_N$ to $I_w^\star$ is a translation carrying $I_w^\star$ to~$I_{f(w)}^\star$.
If $w=1^j 0$ for some $j \geq 0$ then this is a translation by $- \frac{1}{N_0} + \frac{1}{N_j} + \frac{1}{1+N_j}$.
\end{lemma}
\begin{proof}
First we prove this for the special case when $w=1^j 0$. By \autoref{prop:interval formula}, the endpoints of~$I_w$ are 
$$h\left(w \overline{0}\right)=\sum_{i=0}^{j-1} \left(\frac{1}{N_i}-\frac{1}{N_{i+1}}\right)=\frac{1}{N_0}-\frac{1}{N_j},$$
$$h\left(w \overline{1}\right)=\sum_{i=0}^{j-1} \left(\frac{1}{N_i}-\frac{1}{N_{i+1}}\right)+\sum_{i=j+1}^\infty
 \left(\frac{1}{N_i}-\frac{1}{N_{i+1}}\right)=\frac{1}{N_0}-\frac{1}{N_j}+\frac{1}{N_{j+1}}=\frac{1}{N_0}-\frac{1}{1+N_j}.$$
Thus if $x \in I_w^\star$ then we see by definition of $T_N$ that 
\begin{equation}
\label{eq:image}
T_N (x)=x - \frac{1}{N_0} + \frac{1}{N_j} + \frac{1}{1+N_j}.
\end{equation}
The word $f(w)$ is a string of $j$ zeros followed by a one. Thus, we see that the endpoints of $I_{f(w)}$~are
$$h\left(f(w) \overline{0}\right)=\frac{1}{N_j}-\frac{1}{N_{j+1}}=\frac{1}{1+N_j},$$
$$h\left(f(w) \overline{1}\right)=\sum_{i=j}^\infty
 \left(\frac{1}{N_i}-\frac{1}{N_{i+1}}\right)=\frac{1}{N_j}.$$
Observe that these new endpoints differ from the endpoints of $I_w$ found earlier by a translation by
$- \frac{1}{N_0} + \frac{1}{N_j} + \frac{1}{1+N_j}$ which is exactly how $T_N$ acts.
This proves the second statement of the~lemma.

Now suppose that $w' \in \sA^\ast$ is a word which has at least one zero.
As in \eqref{eq:plus one words}, we can then define $j=\min~\{i:~w'_i=0\}$. Hence $w=w'_0 \ldots w'_j$ is a word consisting of $j$ ones followed by a zero, so the previous paragraph implies that $T_N$ restricted to 
$I_w^\star$ is a translation by $- \frac{1}{N_0} + \frac{1}{N_j} + \frac{1}{1+N_j}$. 
Recall that $I_{w'}^\star \subset I_w^\star$ which implies that the restriction of $T_N$ to $I_{w'}^\star$  also acts by the same translation. The intervals $I_{w'}$ and $I_{f(w')}$ have the same length and their left endpoints differ by 
$$h\left(f(w') \overline{0}\right)-h\left(w' \overline{0}\right)=
h\left(f(w) \overline{0}\right)-h\left(w \overline{0}\right)=
- \frac{1}{N_0} + \frac{1}{N_j} + \frac{1}{1+N_j}$$
so that indeed $T_N(I_{w'}^\star)=I_{f(w')}^\star$.
\end{proof}

\begin{theorem}\label{thm:periodic 2}
If $x \in X_N \smallsetminus h(\sA^\N \smallsetminus \sN)$, then $x$ is periodic under $T_N$.
\end{theorem}
\begin{proof}
Let $x \in X_N \smallsetminus h(\sA^\N \smallsetminus \sN)$.
Then either $x$ is not contained in the closed set $\sC$ or we can apply statement (3) of \autoref{prop:N}. In both cases, there is an $\epsilon$ such that $(x,x+\epsilon) \cap \sC=\emptyset$.
Since $X_N=I_\varepsilon^\star$, the interval $(x,x+\epsilon)$ must lie in one of the gaps of the Cantor set, i.e., there is a $w \in \sA^k$ such that 
$$(x,x+\epsilon) \subset I_w \smallsetminus (I_{w0} \cup I_{w1}).$$
It follows that $x \in I_w^\star \smallsetminus (I_{w0}^\star \cup I_{w1}^\star)$.
Then $I_{0^k}^\star$ has the same length as $I_w^\star$ and so we have $I_w^\star=\tau+I_{0^k}^\star$ for some $\tau \in \R$ acting by translation. Set $x_0=x-\tau \in I_{0^k}^\star$. Observe that there exists an $m \geq 0$ such that $f^m(0^k)=w$, where $f$ is as in \eqref{eq:plus one words}. By \autoref{lem:key},
we know that $T_N^m$ restricted to $I_{0^k}^\star$ is a translation carrying this interval $I_{0^k}^\star$ to $I_{w}^\star$. Thus $T_N^m(x_0)=x$. It also follows
that $T_N^m(I_{0^{k+1}}^\star)=I_{w0}^\star$ and $T_N^m(I_{0^k1}^\star)=I_{w1}^\star$ and in particular $x_0 \not \in I_{0^{k+1}}^\star \cup I_{0^k 1}^\star$.

Now observe that $I_{0^k}^\star=X_{N_k}=[0, 1/N_k)$ by \autoref{prop:interval formula}, and by \autoref{cor:return}
the first return map of $T_N$ to this interval is $T_{N_k}$. Since $I_{0^{k+1}}^\star=[0,1/N_{k+1})$ and $I_{0^k 1}^\star=[1/N_k-1/N_{k+1},1/N_k)$, \autoref{cor:periodic interval} tells us that $x_0$ is periodic under $T_{N_k}$ and therefore also periodic under $T_N$. Since $x=T_N^m(x_0)$,
$x$ is also periodic.
\end{proof}

\begin{theorem}
\label{thm:aperiodic 2}
For any $\alpha \in \sA^\N \smallsetminus \sN$, we have $T_N \circ h(\alpha)=h \circ f(\alpha)$. In particular, no point in $h(\sA^\N \smallsetminus \sN)$ has a periodic orbit.
\end{theorem}
\begin{proof}
Fix $\alpha\in \sA^\N \smallsetminus \sN$. Define $j=\min~(\{k:~\alpha_k=0\} \cup \{+\infty\})$ as in \eqref{eq:f}. Since $\alpha \not \in \sN$ we have $j < +\infty$. The initial word of $\alpha$ then has the form $1^j 0$ and the initial word of $f(\alpha)$ is $0^j 1$. The rest of the sequence $f(\alpha)$ agrees with $\alpha$.
Therefore we have 
\begin{equation}
\label{eq:compare}
h \circ f(\alpha)-h(\alpha)= \left(\frac{1}{N_j}-\frac{1}{N_{j+1}}\right)-\sum_{i=0}^{j-1} \left(\frac{1}{N_i} - \frac{1}{N_{i+1}}\right)=
-\frac{1}{N_0} + \frac{1}{N_j}+\frac{1}{1+N_j}.
\end{equation}
Let $x = h(\alpha)$. Then $x \in I_{1^j 0}^\star$ and $T_N$ acts as a translation by $-\frac{1}{N_0} + \frac{1}{N_j}+\frac{1}{1+N_j}$ on $I_{1^j 0}^\star$;
see \autoref{lem:key}. Thus by equation~\eqref{eq:compare} we see that $h \circ f(\alpha)=T_N\circ h(\alpha)$.
Since $f$ has no periodic orbits and $h$ restricted to $\sA^\N\smallsetminus \sN$ is injective, we see that $T_N$ has no periodic orbits in $h(\sA^\N\smallsetminus \sN)$.
\end{proof}

We finish by proving the first two theorems of our article. 

\begin{proof}[Proof of Theorems \ref{thm: zero dimensional} and \ref{thm:aperiodic}]
Recall that $\Lambda_N$ denoted the set of points in $X_N$ with aperiodic orbits under $T_N$. Together \autoref{thm:periodic 2} and \autoref{thm:aperiodic 2} guarantee that $\Lambda_N=h(\sA^\N \smallsetminus \sN)$. 
\autoref{thm:aperiodic 2} then directly implies \autoref{thm:aperiodic}.
Statement (2) in \autoref{prop:N} shows that the closure $\bar \Lambda_N$ is the Cantor set $\sC$.
Further by statement (3) of \autoref{prop:N} we see that $\Lambda_N$ has the form claimed in \autoref{thm: zero dimensional}. The fact that $\Lambda_N$ has Hausdorff dimension zero follows from \autoref{prop:Hausdorff dimension}.
\end{proof}

\section*{Acknowledgements}
The authors acknowledge support from U.S. National Science Foundation grants DMS 1107452, 1107263, 1107367 ``RNMS: GEometric structures And Representation varieties'' (the GEAR Network).
Contributions of the first author are based upon work supported by the National Science Foundation under Grant Number DMS-1500965 as well as a PSC-CUNY Award (funded by The Professional Staff Congress and The City University of New
York).
The second author acknowledges support from NSERC Discovery grant RGPIN 06486.
The work of the third author is partially supported by NSERC grant RGPIN 06521.

\bibliographystyle{amsalpha}
\bibliography{bibliography}
\end{document}